\pdfoutput=1 
\documentclass{article}
\usepackage{amsmath,amssymb,amsthm,color}

\def\R{\mathbb{R}}
\def\N{\mathbb{N}}

\def\C{\mathbb{C}}

\renewcommand{\geq}{\geqslant}
\renewcommand{\leq}{\leqslant}
\newcommand{\Jcal}{\mathcal{J}}

\newtheorem{theorem}{Theorem}[section]

\newtheorem{lemma}{Lemma}[section]
\theoremstyle{definition}
\theoremstyle{definition}\newtheorem{remark}{Remark}[section]

\title{Spiraling of sub-Riemannian  geodesics around the Reeb flow in the 3D contact case}
\date{}

\author{Yves Colin de Verdi\`ere\footnote{Universit\'e de Grenoble-Alpes,
Institut Fourier, Unit{\'e} mixte de recherche CNRS-UJF 5582, BP 74, 38402-Saint Martin d'H\`eres Cedex, France 
(\texttt{yves.colin-de-verdiere@univ-grenoble-alpes.fr}).}
\and
Luc Hillairet \footnote{Universit\'e d'Orl\'eans, Institut Denis Poisson, route de Chartres,
 45067 Orl\'eans Cedex 2, France (\texttt{luc.hillairet@univ-orleans.fr}). }
\and
Emmanuel Tr\'elat\footnote{Sorbonne Universit\'e, CNRS, Universit\'e de Paris, Inria, Laboratoire Jacques-Louis Lions (LJLL), F-75005 Paris, France (\texttt{emmanuel.trelat@sorbonne-universite.fr}).}
}

\begin{document}

\maketitle

\begin{abstract}
We consider a closed three-dimensional contact sub-Riemannian manifold.
The objective of this note is to provide a precise description of the sub-Riemannian geodesics with large initial momenta:
 we prove that they ``spiral around the Reeb orbits'', not only in the phase space but also in the configuration space.
Our analysis is based on a normal form along any Reeb orbit due to Melrose.
\end{abstract}

\section{Introduction and main result}
Let $(M,D,g)$ be a closed sub-Riemannian (sR) manifold of dimension $3$, where $D$ is a contact distribution endowed with a Riemannian metric $g$.
We assume for simplicity that $D$ is oriented, and we denote by $\alpha_g$ the unique one-form defining $D$ so that ${d\alpha _g} _{\vert D} $ is the volume form induced by the metric $g$ on $D$. The Reeb vector field $Z$ is then characterized by the relations $\alpha _g (Z)=1$ and $d\alpha _g(Z,\cdot)=0$.
Equivalently, given any positive $g$-orthonormal local frame $(X,Y)$ of $D$, $Z$ is the unique vector field such that
 $[X,Z]\in D$, $[Y,Z]\in D$ and $[X,Y]=-Z\mod D$.

We recall that the cometric $g^\star :T^\star M \rightarrow \R^+$ associated with the sR metric $g$ is defined by
$$
g^\star (q,p)= \Vert p_{| D(q)} \Vert_{g(q)}^2 .
$$

The Hamiltonian $G=\sqrt{g^\star}$, which is homogeneous of degree $1$, generates the sR geodesic flow: the projections onto $M$
 of the integral curves of the associated Hamiltonian vector field $\vec{G}$ are the sR geodesics with speed $1$.
Note that the function $G$ is not smooth along the line bundle $\Sigma=G^{-1}(0)=D^\perp$ (the annihilator of $D$).
The geodesic flow $G_t =\exp(t\vec{G})$ is homogeneous of degree $0$, and thus is defined and smooth on $S^\star M \setminus S\Sigma$.
Here, $S^\star M$ is the unit cotangent bundle, and $S\Sigma$ is the sphere bundle of $\Sigma$ (quotient of $\Sigma $ 
by positive homotheties).

The Reeb vector field $Z$ has the following dynamical interpretation. If $v_0\in D(q_0)$, then there exists a one-parameter family
 of geodesics associated with the Cauchy data $(q_0,v_0)$, all of them being the projections of the integral curves of $\vec G$ with
 Cauchy data $(q_0, p_0)$ and $( p_0) _{\vert D_{q_0}}=g(v_0,\cdot)$. For every $s\in\R$, the projections on $M$ of the integral curves
 of $\vec G$ with Cauchy data $(q_0, p_0+s\alpha_g )$ in the cotangent space have the same Cauchy data $(q_0,v_0)$ in the tangent space.
 As $s\rightarrow \pm\infty $, they spiral around the integral curves of $\mp Z$. 

In this paper, our objective is to give a precise description of these sR geodesics with large initial momenta: we will explain
 in what precise sense they spiral around the Reeb orbits. 

To establish this feature, we will use an exact local normal form along an arc of a Reeb orbit. This result is originally due to
 Melrose (see \cite[Proposition 2.3]{Me-84}); his proof is however rather sketchy and we will give here a full proof of it (which
 is far from being obvious) in Section \ref{sec:Melrose}.

In Section \ref{sec:spiraling}, we will deduce from this normal form the spiraling property of sR geodesics around Reeb trajectories,
 not only in the phase space $T^\star M$ (the cotangent bundle of $M$), but also in the configuration space $M$. 
We will then be in a position to explain more precisely what \emph{spiraling} means.

\section{Melrose normal form along a Reeb orbit}\label{sec:Melrose}
We denote by $\omega$ the canonical symplectic form of $T^\star M$. In local symplectic coordinates $(q,p)$, we have $\omega=dq\wedge dp$ (see Appendix \ref{app:symplectic} for the sign conventions that are used here and throughout).
The codimension-$2$ manifold $\Sigma$ is a symplectic subcone of $T^\star M$, endowed with the restriction $\omega_{\vert\Sigma}$. Note that
$$
\Sigma = \{ (q,s\,\alpha_g(q))\in T^\star M\ \mid\ q\in M, s\in\R\}.
$$
We define the Hamiltonian $\rho :\Sigma \rightarrow \R$ by $\rho (s\,\alpha _g)=|s|$. The projection onto $M$ of the Hamiltonian vector field $\vec{\rho}$ is $\pm Z$, depending on the sign of $s$. We restrict our study to $\Sigma^+=\rho^{-1}((0,+\infty))$, the cone of positive multiples of $\alpha _g $. The cone  $\Sigma ^-=-\Sigma ^+ $ is obtained by changing the orientation of $D$. 
Given an open subset $U\subset M$, we denote by $\Sigma ^+_U$ the cone 
$\Sigma ^+_U= \{ (q,s\alpha_g(q) )\ \mid\ q\in U, s>0 \}$.

We consider the symplectic conic manifold $\Sigma ^+\times \R^2_{u,v}$ endowed with the symplectic
 form $\tilde\omega = \omega_{\vert\Sigma} + du \wedge dv$ and with the conic structure defined by 
 $\lambda \cdot (q,s\,\alpha_g(q),u,v)=(q,\lambda s\,\alpha_g(q), \sqrt{\lambda }u, \sqrt{\lambda }v)$ for any $\lambda>0$.
We define the function $I$ on $\Sigma ^+\times \R^2_{u,v}$ by $I(\sigma,u,v)=u^2 + v^2$, for any $\sigma\in\Sigma^+$ and $(u,v)\in\R^2$.
 is the Hessian of the function $(u,v)\mapsto g^\star(\sigma,u,v)$ (which vanishes as well as its differential at $(0,0)$) for any fixed
 $\sigma\in\Sigma^+$.

We are going to establish the following Melrose normal form.

\begin{theorem}\label{theo:Melrose-NF} 
Let $\Gamma _0 $ be a closed arc of a Reeb orbit in $M$, diffeomorphic to $[0,1]$.
There exist a neighborhood $U$ of $\Gamma _0 $ and a homogeneous symplectic diffeomorphism $\chi$ of a conic neighborhood $C$ of $\Sigma^+_U $ in $T^\star U$ to a conical neighborhood $C'$  of $\Sigma^+ _U \times \{0\}$ in $\Sigma^+ _U \times \R^2$, satisfying $\chi (\sigma)=(\sigma ,0) $ for every $\sigma\in\Sigma^+_U $, such that $g^\star\circ\chi^{-1} = \rho I$. 
\end{theorem}

The proof in done in Sections \ref{sec:Birkhoff} and \ref{sec:Nelson} hereafter. But Section \ref{sec:Birkhoff} also contains
 results on automorphisms preserving the normal form, and on a parallel transport property along Reeb trajectories.

\subsection{Birkhoff normal form and parallel transport}\label{sec:Birkhoff}
Let us first recall the Birkhoff normal form derived in \cite{CHT-I}.
In what follows, given $k\in\N\cup\{+\infty\}$, given smooth maps $f_1$ and $f_2$ on $T^\star M$ (or on $\Sigma ^+\times \R^2_{u,v}$)
 with values in some manifold, the notation $f_1=f_2+\mathrm{O}_\Sigma(k)$ means that $f_1$ coincides with $f_2$ along $\Sigma$
 (or along $\Sigma\times\{0\}$) at order $k$, at least. 

\begin{theorem}[\cite{CHT-I}]\label{theo:BNF} 
Let $q_0\in M$ be arbitrary. 
There exist a conic neighborhood $C$ of $(q_0, \alpha_g(q_0))\in\Sigma^+$ in $T^\star M\setminus\{0\}$ and a smooth homogeneous symplectomorphism $\chi:C \rightarrow \chi(C)\subset \Sigma ^+\times \R^2_{u,v}$, satisfying $\chi (\sigma )=(\sigma, 0)$ for every $\sigma \in \Sigma^+ \cap C$, such that $g^\star\circ\chi^{-1} = \rho I  + \mathrm{O}_\Sigma(\infty)$.
\end{theorem}

In other words, in local coordinates $(q,s,u,v)$ as above, we have $g^\star\circ\chi^{-1}(q,s,u,v) = s(u^2+v^2)  + \mathrm{O}_\Sigma(\infty)$.
Note that the proof of Theorem \ref{theo:BNF}, done in \cite{CHT-I}, is quite long and technical, and consists in deriving a normal form
 with remainder terms, in using the Darboux-Weinstein lemma and then in improving the remainder terms to a flat remainder 
term $\mathrm{O}_\Sigma(\infty)$ by solving an infinite series of cohomological equations. The final step of the proof of
 Theorem \ref{theo:Melrose-NF} will in particular consist of removing the flat term $\mathrm{O}_\Sigma(\infty)$
 (see Section \ref{sec:Nelson}).

Now, we give a result on symplectomorphisms (canonical transforms) preserving the above Birkhoff normal form.
 We identify $\R^2_{u,v}\sim\C$ with $(u,v)\sim u+iv$ for convenience.

\begin{theorem}\label{theo:LS}
With the notations of Theorem \ref{theo:BNF}, let $\Theta:\chi(C) \rightarrow \chi(C)$ be a smooth homogeneous symplectomorphism
 (for the symplectic form $\tilde\omega$), satisfying $\Theta(\sigma,0)=(\sigma,0)$ for every $\sigma\in\Sigma^+$, such
 that $g^\star\circ\chi^{-1}\circ\Theta^{-1} = \rho I  + \mathrm{O}_\Sigma(\infty)$, that is, such that $\rho I\circ\Theta = \rho I + \mathrm{O}_\Sigma(\infty)$.
Then there exist smooth mappings $F:C\rightarrow\Sigma^+$ satisfying $F(\sigma,0)=\sigma$ for every $\sigma\in\Sigma^+$,
 and $\theta_j:\Sigma^+\cap C\rightarrow\R$ such that $\{\rho,  \theta _j \}_\omega=0$ for every $j\in\N$, i.e., all
 functions $ \theta_j $ are constant along the Reeb orbits, and such that
$$
\Theta (\sigma , u+iv)= \left( F(\sigma ,u+iv), e^{i \sum _{j=0}^{+\infty} \theta _j (\sigma )I(u,v)^j }(u+iv) \right) + \mathrm{O}_\Sigma (\infty ) ,
$$
for all $(\sigma,u,v)\in C$. 
In particular, not only $\rho I$ but also $I$ and $\rho$ are preserved by $\Theta$ modulo $\mathrm{O}_\Sigma (\infty )$.
\end{theorem}

Theorem \ref{theo:LS} actually follows from the Lewis-Sternberg theorem (see \cite{Le-39,St-50}, see also \cite{IS-02}),
 or more precisely from a variant of it where we assume that $\chi$ is the identity along $\Sigma^+$. But we provide hereafter a
 more direct proof.

\begin{proof}[Proof of Theorem \ref{theo:LS}.]
We start with the following lemma. Recall that $(\Sigma^+,\omega_{\vert\Sigma})$ is a symplectic conic manifold of dimension $4$.

\begin{lemma}\label{lemm:cano}
Let $N\geq 2$ be an arbitrary integer.
Let $\Theta_N:\chi(C)\rightarrow\chi(C)$ be a smooth symplectomorphism (for the symplectic form $\tilde\omega$), satisfying
 $\Theta_N = \mathrm{id} + R_N + \mathrm{O}_\Sigma (N+1)$, with $R_N$ homogeneous of degree $N\geq 2$ in $(u,v)$.
 Then there exists a smooth function $S: \chi(C) \rightarrow \R$, homogeneous of degree $N+1$ in $(u,v)$, such
 that $\Theta_N = \exp(\vec{S}) + \mathrm{O}_\Sigma (N+1)$.
\end{lemma}

\begin{proof}[Proof of Lemma \ref{lemm:cano}.]
Let us choose local symplectic coordinates $(q,p)$ with $q=(q_1,q_2)$, $p=(p_1,p_2)$  so that
$\Sigma =\{ q_2=p_2=0 \} $.
We write
$\Theta _N (q,p)=(q_1, q_2+A_N (q,p), p_1+B_N (q,p ) ) + O_\Sigma (N+1)$.
Using the fact that $\Theta _N$ is symplectic, we get
$dA_N \wedge dp - dB_N \wedge dq =O_\Sigma (N)$.
Hence, there exists $S$ homogeneous of degree $N+1$ in the variable $(q_2,p_2)$ so that
$dS = A_N dp -B_N dq + O_\Sigma (N+1) $. This gives the result, because
${\rm exp}(\vec{S})(x)= x+ \vec{S}(x)+ O_\Sigma (2N) $.

\end{proof}

Hereafter, we use local coordinates $(\sigma,u,v)$, with $\sigma=(q,s)$. In order to avoid heavy notations, we denote
 (without any index) by $\{\ ,\ \}$ the Poisson bracket with respect to the symplectic form
 $\tilde\omega = \omega_{\vert\Sigma} + du\wedge dv$. Note that the variables $\sigma$ and $(u,v)$ are symplectically orthogonal,
 and that the coordinates $u$ and $v$ are symplectically conjugate and $\{u,v\}=1$.

In these local coordinates $(q,s,u,v)$, we have $g^\star\circ\chi^{-1}(q,s,u,v)=s(u^2+v^2)+\mathrm{O}_\Sigma(\infty)$. Setting
 $\Theta(q,s,u,v)=(q',s',u',v')$, we have also
 $g^\star\circ\chi^{-1}\circ\Theta^{-1}(q',s',u',v')=s'({u'}^2+{v'}^2)+\mathrm{O}_\Sigma(\infty)$.
 It follows that $s'({u'}^2+{v'}^2) = s(u^2+v^2) +\mathrm{O}_\Sigma(\infty)$, which can be written in short as $\rho I\circ\Theta=\rho I+\mathrm{O}_\Sigma(\infty)$.

Besides, since $\Theta(\sigma,0)=(\sigma,0)$ and the Hamiltonian vector field of $\rho $
is preserved,  we infer that $\sigma'=\sigma+\mathrm{O}_\Sigma(2)$.
Therefore we get that ${u'}^2+{v'}^2 = u^2+v^2 +\mathrm{O}_\Sigma(4)$. As a consequence, there exists $\theta_0(\sigma)\in\R$
 such that $u'+iv'=e^{i\theta_0(\sigma)}(u+iv) +\mathrm{O}_\Sigma(3)$, where $\sigma=(q,s)$. Defining the smooth $2$-homogeneous
 function $S_2 = \theta _0 I /2$, i.e., $S_2(\sigma,u,v) = \theta_0(\sigma) (u^2+v^2)/2$, the corresponding Hamiltonian vector
 field $\vec S_2$ generates a rotation of angle $\theta_0(\sigma)$ in the coordinates $(u,v)$, and we infer that
$$
\Theta = \exp(\vec S_2) + \mathrm{O}_\Sigma(3).
$$
Now, writing $\exp(-\vec S_2)\circ\Theta = \mathrm{id} + R_2 + \mathrm{O}_\Sigma(4)$ with $R_2$ homogeneous of degree $3$ in $(u,v)$,
 applying Lemma \ref{lemm:cano} with $N=3$ yields the existence of a smooth $4$-homogeneous function $S_3$ such that 
\begin{equation}\label{eqTheta1}
\Theta = \exp(\vec{S_2})\circ \exp(\vec{S_3}) + \mathrm{O}_\Sigma(4) .
\end{equation}
Since $S_2 = \theta _0 I /2$, we have $\{S_2,\rho I\} = \frac{1}{2} \{\theta_0,\rho\}I^2$ and $\{S_2,\{S_2,\rho I\}\} = \frac{1}{4}
 \{\{\theta_0,\{\theta_0,\rho\}\}I^3=\mathrm{O}_\Sigma(6)$, and thus
$$
\rho I\circ \exp(\vec{S_2}) = \exp(\{S_2,\cdot\}).\rho I = \rho I + \frac{1}{2} \{\theta_0,\rho\}I^2 + \mathrm{O}_\Sigma(6) .
$$
Now, $S_3(\sigma,u,v)$ is a sum of terms of the kind $a_3(\sigma)Q_3(u,v)$ with $a_3$ smooth and $Q_3$ homogeneous in $(u,v)$ of
 degree $3$ (running over $u^3$, $u^2v$, $uv^2$, $v^3$). Taking one of them, $S_3(\sigma,u,v)=a_3(\sigma)Q_3(u,v)$, we compute
\begin{multline}\label{eeqTheta1}
\rho I\circ \exp(\vec{S_2})\circ \exp(\vec{S_3}) = \exp(\{S_3,\cdot\}).\left( \rho I + \frac{1}{2} \{\theta_0,\rho\}I^2 +
 \mathrm{O}_\Sigma(6) \right)  \\
= \rho I + \frac{1}{2} \{\theta_0,\rho\}I^2  + a_3\rho\{ Q_3, I\} + \mathrm{O}_\Sigma(5) .
\end{multline}
Computing $\{u^3,I\}=6u^2v$, $\{u^2v,I\}=-2u^3+4uv^2$, $\{uv^2,I\}=-4u^3v+2v^3$, $\{v^3,I\}=-6uv^2$, we note that $\{ Q_3, I\}$ is
 homogeneous of degree $3$.
Now, since we must have $\rho I\circ\Theta=\rho I+\mathrm{O}_\Sigma(\infty)$, it follows from \eqref{eqTheta1} that
 $\rho I = \rho I\circ\exp(\vec{S_2})\circ \exp(\vec{S_3}) + \mathrm{O}_\Sigma(5)$, and using \eqref{eeqTheta1} we infer that
$\frac{1}{2} \{\theta_0,\rho\}I^2  + a_3\rho\{ Q_3, I\} = \mathrm{O}_\Sigma(5)$. 
Since $I^2$ is $4$-homogeneous and $\{ Q_3, I\}$ is $3$-homogeneous, it follows that $\{\theta_0,\rho\}=0$ and $a_3\{ Q_3, I\}=0$.
Then we have either $a_3=0$ and thus $S_3=0$, or $\{ Q_3, I\}=0$, and in this case, it follows from Remark \ref{rem:linearalgebra}
 (in Appendix \ref{app:linearalgebra}) that $Q_3=0$. Then, in all cases, we have $S_3=0$.

Since $S_3=0$, writing now $\exp(-\vec S_2)\circ\Theta = \mathrm{id} + \mathrm{O}_\Sigma(3) = \mathrm{id} + R_3 + \mathrm{O}_\Sigma(4)$ 
with $R_3$ homogeneous of degree $3$, applying again Lemma \ref{lemm:cano} with $N=3$ yields the existence of a smooth $4$-homogeneous 
function $S_4$ such that
$$
\Theta = \exp(\vec{S_2})\circ \exp(\vec{S_4}) + \mathrm{O}_\Sigma (4).
$$
As above, setting $Q_4(\sigma,u,v)=a_4(\sigma)Q_4(u,v)$, we compute $\rho I\circ \exp(\vec{S_2})\circ \exp(\vec{S_4}) = \rho I + a_4\rho \{ Q_4,I \} + \mathrm{O}_\Sigma(6)$, and similarly, since we must have $\rho I\circ\Theta=\rho I+\mathrm{O}_\Sigma(\infty)$, it follows that $a_4\rho \{ Q_4,I \}=0$. Then, either $a_4=0$ or $\{ Q_4,I \}=0$. In the latter case, by Remark \ref{rem:linearalgebra} we have $AQ_4=0$, and by Lemma \ref{rem:linearalgebra} we get that $Q_4\in P_4^{\textrm{inv}}$, that is, $Q_4=c I^2$ for some $c\in\R$.

\medskip

The reasoning continues by iteration, and is even simpler now that we know that $\{\theta_0,\rho\}=0$ and thus that $\rho I\circ\exp(\vec S_2)=\rho I+\mathrm{O}(\infty)$. 
In passing, we note that this implies immediately that $I\circ\exp(\vec S_2)=I+\mathrm{O}(\infty)$ and $\rho\circ\exp(\vec S_2)=\rho+\mathrm{O}(\infty)$
We only describe the next step, then the recurrence is immediate.

Writing $\exp(-\vec S_4)\circ\exp(-\vec S_2)\circ\Theta = \mathrm{id} + R_4 + \mathrm{O}_\Sigma(5)$ with $R_4$ that is $4$-homogeneous, applying Lemma \ref{lemm:cano} with $N=4$ yields the existence of a smooth $5$-homogeneous function $S_5$ such that 
$$
\Theta = \exp(\vec{S_2})\circ \exp(\vec{S_4})\circ \exp(\vec{S_5}) + \mathrm{O}_\Sigma (5).
$$
Taking $S_4(\sigma,u,v)=a_4(\sigma)Q_4(u,v)$, we compute $\rho I\circ \exp(\vec{S_2})\circ \exp(\vec{S_4})\circ \exp(\vec{S_5}) = \rho I + \{ a_4,\rho\} I^3 + a_5\rho \{ Q_5,I \} + \mathrm{O}_\Sigma (7)$. Therefore, reasoning as above, we infer that $\{ a_4,\rho\}=0$ and $a_5\rho \{ Q_5,I \}=0$. Then either $a_5=0$ or $\{ Q_5,I \}=0$. In the latter case, we get $Q_5=0$ by Remark \ref{rem:linearalgebra}. Hence $S_5=0$.

Since $S_5=0$, writing now $\exp(-\vec S_4)\circ\exp(-\vec S_2)\circ\Theta = \mathrm{id} + \mathrm{O}_\Sigma(5) = \mathrm{id} + R_5 + \mathrm{O}_\Sigma(6)$ with $R_5$ homogeneous of degree $5$, applying again Lemma \ref{lemm:cano} with $N=5$ yields the existence of a smooth $6$-homogeneous function $S_6$ such that
$\Theta = \exp(\vec{S_2})\circ \exp(\vec{S_4})\circ \exp(\vec{S_6}) + \mathrm{O}_\Sigma (6)$.
Reasoning as for $S_4$, we establish that $S_6(\sigma,u,v) = a_6(\sigma) I(u,v)^3$.

A recurrence argument concludes the proof.
\end{proof}

\paragraph{Parallel transport along Reeb trajectories.}\label{par:par}
An interesting consequence of Theorem \ref{theo:LS} is that this result reveals a property of parallel transport of $D$ along Reeb trajectories: this transport is the image by the differential 
of $\Pi \circ \Theta^{-1} $ of the trivial transport in the $(u,v)$ planes.
Indeed, Theorem \ref{theo:LS} implies that $d\Theta(\sigma, 0 )$  is a rotation in the $(u,v)$ plane with constant angle $2s_1(\sigma )$ along the orbits of $\vec{\rho}$.
Projecting this property onto $M$ gives the invariance of the parallel transport under the change of charts preserving the Birkhoff normal form.

\subsection{End of the proof of Theorem \ref{theo:Melrose-NF} by Nelson's trick}\label{sec:Nelson}
To end the proof and in particular to remove the flat remainder term $\mathrm{O}_\Sigma(\infty)$ in the Birkhoff normal form given in Theorem \ref{theo:BNF}, we use the nice \emph{scattering} method due to Edward Nelson (see \cite{Nelson}), thanks to which we will establish an exact normal form in a Reeb flow box. 

The Birkhoff normal form given in Theorem \ref{theo:BNF} is valid in some conic neighborhood of $\Sigma ^+_U$, where $U $ is a flow box for the Reeb flow, i.e., $U$ is diffeomorphic to $S_x \times (-2c , T_0 +2c )_y$ for some $c>0$, with $Z=\partial _y$.
Theorem \ref{theo:BNF} provides us with a homogenous symplectic diffeomorphism that we use as a coordinate system, so that we now work in $\Sigma^+_U\times \R^2_{u,v}$. 

We define the Hamiltonian $H=\sqrt{\rho I }$ on $\Sigma ^+ \times (\R^2  \setminus \{0\}).$ We set $J=\sqrt{I/\rho} $ and we define $\theta$ so that $(J,\theta)$ are polar coordinates in $\R^2_{u,v}$. The function $J$ is homogeneous of degree $0$ and is a measure of the angular distance to $\Sigma$. 
The sphere bundle $S(\Sigma ^+_U \times (\R^2  \setminus \{0\}))$ is thus parametrized by $(m,J,\theta )=[ m,\alpha_g(m);J,\theta )]$ with $m\in U$. 

We denote by $H_t$ the flow of the Hamiltonian vector field
$$
\vec{H}=\frac{J}{2}\vec{\rho} + \frac{1}{J} \frac{\partial }{\partial \theta } ,
$$
which is homogeneous of degree $0$ and hence is defined on the sphere bundle. It is explicitly given by 
$$
H_t (m,J,\theta)= ( \mathcal{R}_{Jt/2}(m), J , \theta +  t/J ),
$$
where   $(\mathcal{R}_s)_{s\in\R}$ is  the Reeb flow on $M$.

Let $V_0 \subset V$ be defined  by $V_0=S_0 \times (-c , T_0 +c )$ with $\bar{S_0}\subset S $.
Let us choose $0< a_0< a_1$ and  denote by $C_j$ the cones  $C_{j}= \{ (m,J, \theta )\ \mid\ m\in U_j, J < a_j  \} $.
Using the Birkhoff normal form, recalling that $G=\sqrt{g^\star}$, we have $G^2= H^2+ \mathrm{O}_\Sigma(\infty)$, where the remainder term $\mathrm{O}_\Sigma(\infty)$ is smooth. This can be rewritten as
$$
G^2 = H^2(1+\mathrm{O}_\Sigma(\infty)) ,
$$
so that the geodesic Hamiltonian 
satisfies $G=H+ \mathrm{O}_\Sigma(\infty )$ in $C_{1}$ and the remainder term $\mathrm{O}_\Sigma (\infty )$ in the latter equation is smooth even though $G$ and $H$ are not.

We now extend $G$  to $\Sigma ^+ \times (\R^2\setminus \{ 0\})$
as follows: let $\psi $ be a smooth function that is homogeneous of degree $0$ on $\Sigma ^+ \times \R^2$ identically equal to $1$ in $C_{0}$ and identically equal to $0$ outside of $C_{1}$.
We define $\tilde{G}=\psi G +(1-\psi)H$, and we check that $R=\tilde{G}-H =\mathrm{O}_\Sigma (\infty )$ and that $\{ J, \tilde{G} \} =\dot{\mathrm{O}}_\Sigma (\infty )$, where the upper dot indicates that we are dealing here with functions on $\Sigma^+\times (\R^2\setminus \{0\})$. 
It follows from the definition that the flows 
$\tilde{G}_t $ and $H_t$ are complete and coincide outside of $C_1$. 

\begin{lemma}\label{lemm:firstint} 
Given any $z_0$, we set $\Jcal(t)=J(\tilde{G}_t (z_0))$. For every $N\geq 2$, there exist  $C_N>0$ and $D_N >0$ such that, if $0< \Jcal(0) \leq  \frac{1}{2}$ and $0\leq t \leq C_N/\Jcal(0) ^{2N} $, then $\Jcal(t)\leq 2\Jcal(0)$ and $|\Jcal(t)-\Jcal(0)|\leq D_N  \Jcal(0) ^{2N+1}  \, t$.
\end{lemma}

\begin{proof}[Proof of Lemma \ref{lemm:firstint}.]
Since $\{ J, \tilde{G}\} = \dot{\mathrm{O}}_{\Sigma}(\infty)$ is homogenous of degree $0$, for any $N$ there exists $C>0$ such that 
$$
0\,<\,J(z)\,\leq\, 1 \Rightarrow 
\vert \{ J, \tilde{G}\} (z)\vert \leq C J(z)^{2N+1}.  
$$
It follows that $\dot{\Jcal}(t) \leq C \Jcal^{2N+1}(t)$ as long as
$\Jcal(t)\leq 1.$ 
Integrating, we get 
$$
\Jcal(t)\leq \frac{\Jcal(0)}{\left( 1- 2NC t \Jcal^{2N}(0) \right)^{1/2N}} ,
$$
as long as $(2NC  \Jcal^{2N}(0))t\leq 1/2$ and $\Jcal(t)\leq 1$. Hence, there exists $C_N>0$ (e.g., $C_N= \frac{2^{2N}-1}{2^{2N+1}NC}$)
such that 
$$
0\leq t \leq  C_N \frac{1}{\Jcal(0)^{2N}}\quad\textrm{and}\quad\Jcal(0)\leq
\frac{1}{2} \Rightarrow \Jcal(t) \leq 2 \Jcal(0).
$$
Therefore $J(t)\leq 2J(0)$. Using that $|J't)|\leq  C J^{2N+1}(t) \leq C (2J(0))^{2N+1}$, we get the second estimate (with $D_N= Z^{2N+1}CC_N$).
\end{proof} 

Using the fact that $\vec{\tilde{G}} - \vec{H}= \mathrm{O}_\Sigma (\infty )$, taking $a_1$ small enough, we assume that, in the decomposition $\vec{\tilde{G}}(z)= a\partial _y + V + b\partial _u + c\partial _v$, where $V $ is tangent to $ S_1 \times\{ y \} $, we have $a \geq J/4$ as long as $z$ is in $C_1$.

Lemma \ref{lemm:firstint} then implies that, for $\varepsilon>0$ small enough, if $t\leq \varepsilon (\Jcal(0))^{2N}$ then $\Jcal(t)\geq \frac{\Jcal(0)}{2}$ so that the flow $\tilde{G}_t$ is going out of $C_{1} $ within a time of order at most $\mathrm{O}(1/\Jcal(0))$ in both time directions.

\medskip

Following the method of Nelson, let us now define $\chi :\Sigma ^+ \times (\R^2\setminus \{0\}) \rightarrow  \Sigma ^+ \times (\R^2\setminus \{0\})$ by
$$
\chi (z)= \lim _{t\rightarrow  +\infty } (H_t \circ \tilde{G}_{-t})(z) .
$$
It is well defined and the limit is obtained within a time $\mathrm{O}(1/J(0)$, because as soon as $t_0$ is such that $G_{-t_0}(z)$ has left $C_1$, we have $H_t\circ \tilde{G}_{-t}(z) = H_{t_0} \circ H_{t-t_0}\circ G_{-t+t_0}\circ G_{-t_0}$ for $t>t_0$, and the flows of $H$ and $\tilde{G}$ coincide outside of $C_1$.
  
By definition we have $\chi \circ \tilde{G}_{t} =H_t$, so that $\chi\circ {G}_{t}(z)=H_t(z) $ for $t$ small and $z \in C_{0}$, since the flows of $G$ and $\tilde{G}$ coincide there.

Moreover, $\chi$ is a symplectomorphism of $\Sigma ^+ \times \R^2\setminus \{0\}$ whose inverse is given by $\chi^{-1}=  \lim _{t\rightarrow +\infty } (\tilde{G}_t \circ H_{-t})(z)$.

\medskip

Let us finally prove that we can extend $\chi $ to  $\Sigma ^+\times \R^2 $ by the identity on $\Sigma$ and obtain a smooth symplectomorphism that transforms $G$ to the desired Melrose normal form in $C_0$. The latter fact follows from the definition so that the only issue is the smooth extension. 

Let $\phi$ be a smooth function that is homogeneous of degree $0 $ on $\Sigma ^+ \times \R^2$. Following \cite[Section 3]{Ra-VN-13}, we compute the time derivative of $\mathcal{D}(t)(z)= \phi (H_t \circ \tilde{G}_{-t} (z) )$, given by
$$
\mathcal{D}'(t)= -\{ \tilde{G} ,  \mathcal{D}\} + \{ H\circ \tilde{G}_{-t} ,  \mathcal{D}    \}
=-\{ R\circ \tilde{G}_{-t} ,\mathcal{D} \}
=-\{  R , \phi \circ H_t \} \circ \tilde{G}_{-t} ,
$$
where we have used, successively, the invariance of $\tilde{G} $ under $\tilde{G}_t$, and the
 invariance of the Poisson bracket under $\tilde{G}_t$.
Using the explicit expression for $H_t$ and the fact that $J$ is preserved by the flow of $H_t$, we see that, as long as $J(z) \times t$ is bounded above, the differential of $H_t$ at $z$ is $\mathrm{O}(J(z)^{-2})$. It follows that the Poisson bracket $\{  R , \phi \circ H_t \}$ is $\mathrm{O}(J(z)^\infty)$ as long as $J(z)t$ is bounded. Using Lemma \ref{lemm:firstint}, it follows that $\mathcal{D}'(t)= \mathrm{O}(\Jcal(0)^\infty )$, and thus $| \mathcal{D}(t)-\mathcal{D}(0)|=\mathrm{O}(\Jcal(0)^\infty )$ as long as $\Jcal(0)t$ is bounded.

Locally in $z,$ in the expression giving $\chi(z)$ instead of the limit $t\rightarrow +\infty$ we can fix $t=T_0$ chosen of the order of $1/J(z)$. Then, we have $| \mathcal{D}(t)-\mathcal{D}(0)|=\mathrm{O}(\Jcal(0) ^\infty )$ for $t =\mathrm{O}(1/J(0))$. Choosing $\phi$ among a finite set of functions that give local coordinates in $C_1$, we obtain that $d (\chi(z), z) = \mathrm{O}(J^\infty )$ for some distance $d.$
We have then $d(z_0, \chi (z) )\leq d(z_0, z) + d(z, \chi(z)$ if $z_0 \in \Sigma \times \{0\}$.
This proves that $\chi$, extended by the identity on $\Sigma ^+ \times 0$, is continuous.

To finish, Nelson's trick consists of constructing extensions of these flows to the $k-$jets spaces. For example, if $k=1$, we consider the lifts of both flows to the tangent space $T(\Sigma^+ \times \R^2)$. The same properties as above are satisfied, and we get that $\chi $ can be extended smoothly and its differential is the identity on $\Sigma ^+ \times\{0\}$.
Hence $\chi$ can be extended to a diffeomorphism such that $\tilde{G}\circ \chi = H$ and, in particular, $G \circ \chi (z)= H(z) $ if $z\in C_{0}$. 
Theorem \ref{theo:Melrose-NF} is proved.

\section{Spiraling along  periodic  Reeb orbits}\label{sec:spiraling}
In this section, we show how the Melrose normal form along a Reeb orbit given in Theorem  \ref{theo:Melrose-NF} can be used to describe the spiraling property of geodesics around Reeb trajectories.

Considering the symplectomorphism $\chi$ given by Theorem  \ref{theo:Melrose-NF}, the geodesics of large initial momenta are images under $\chi ^{-1}$ of the curves 
$$
t \rightarrow \left( \mathcal{R}_{tJ_0/2} , J_0 e^{it/J_0 + \theta _0 } \right) ,
$$
with $t=\mathrm{O}(1/J_0)$.
In order to describe these images in a precise way, it is useful to compute the differential $\Xi $  of $\chi ^{-1}$ along $\Sigma ^+ \times \{0\}$.
Fixing a positive $g$-orthonormal basis $(X,Y)$ of $D$ which is \emph{parallel} along the Reeb flow
for the parallel transport defined in \ref{par:par}, we have
$$
\Xi_\sigma ( \delta \sigma, \delta u, \delta v)=  \delta \sigma + \frac{1}{\sqrt{\rho (\sigma)}}( \delta u \, \vec h_X+  \delta v \, \vec h_Y ) .
$$ 

To state the result below, it is more convenient to define a complex structure on $D$: the product by $i$ is defined as the rotation of angle $\pi /2$ with respect to the orientation of $D$ and the metric $g$.

\begin{theorem}\label{theo:spiraling}
Let $q_0\in M$ be arbitrary, and let $(q_0,p_0)\in T^\star_{q_0} M$ be the Cauchy data of a geodesic $t\mapsto\gamma (t)$ starting at $q_0$ with unit speed $\dot{\gamma }(0)=X_0 \in D(q_0)$. We assume that $h_0=h_Z (p_0) \gg 1$ (large initial momentum).

Then, there exists a point $Q_0=Q_0(q_0,p_0)\in M$ close to $q_0$, and a vector $Y_0 \in D_(Q_0)$ close to $X_0$, such that, denoting  by $\Gamma (\tau )= \mathcal{R}_{\tau }(Q_0)$ the Reeb orbit of $Q_0$, and by $Y(t)$ the parallel transport of $Y_0$ along $\Gamma $, 

we have, using the complex structure on $D$, for $t=\mathrm{O}(h_0)$, 
$$
\gamma (t)=  \Gamma (J_0t/2)-i J_0 e^{it/J_0}Y(J_0t/2)) + \mathrm{O}(J_0^2),
\qquad
\dot{\gamma}(t)=   e^{it/J_0} Y(J_0t/2)+  \mathrm{O}(J_0) ,
$$
with $J_0=h_0^{-1}+ \mathrm{O}(h_0^{-3}) $ and $Q_0= q_0 - i h_0^{-1}X_0  + \mathrm{O}(h_0^{-2})$.  
\end{theorem}

In other words, this result says that, on time intervals of length of the order of $1/h_0$, the geodesics spiral along orbits of the Reeb flow taken for time intervals $\mathrm{O}(1)$. 

\begin{proof}
We have $g^\star (q_0,p_0 )=1$ and thus $J_0=1/\rho _0$. Using that $h_Z= \rho + \mathrm{O}_\Sigma (2)$, we get that $J_0=h_0^{-1}+ \mathrm{O}(h_0^{-3})$. Setting  $z(t)=G_t(q_0,p_0)$, we have
$$
z(t)= \chi^{-1}\left(\mathcal{R}_{J_0t/2}(Q_0), \frac{1}{J_0}\alpha _g (Q_0),J_0e^{it/J_0} \right).
$$
This defines $Q_0$ and we have $\gamma (t)= \pi(z(t))$, where $\pi:T^\star M \rightarrow M $ is the canonical projection. The result now follows by taking (if necessary) a covering of the Reeb trajectory by a finite number of charts in which the normal form is valid.
\end{proof}

\section{A conjecture on periodic geodesics}
We consider a periodic orbit $\Gamma$ of the Reeb flow on $M$ of primitive period $T_0>0$.
In this section, using the normal form, we derive an approximation of the return Poincar\'e first return map of the geodesic flow, for geodesics spiraling around $\Gamma$ and that are almost closed within time $T_0/J_0$.
Using this, we give a conjecture on the lengths of long closed geodesics close to $\Gamma $.

Along $\Gamma$, the fiber bundle $D$ is trivial. We then consider a $g$-orthonormal frame of $D(\Gamma (0))$, that we transport in a parallel way along $\Gamma$. 
In this way, we obtain a monodromy of angle $\alpha _0$.
Let us consider geodesics spiraling around $\Gamma $ that are close to $\gamma_0(t) = \mathcal{R}_{J_0t/2}(\Gamma (0))-i J_0e^{it/J_0} Y(J_0t/2)$.
The conditions for $\gamma _0 $ to be a periodic integral curve of $\overrightarrow{\rho I} $ of
period $T_{j,k}$, covering  $j$ times $\Gamma$ and winding $k$ times around $\Gamma$, are given by $J_0T_{j,k}=2jT_0$ and $\frac{T_{j,k}}{J_0} + j\alpha_0 =2k\pi$.
It follows that
$$
T_{j,k} = 2\sqrt{jk\pi T_0  \left(1-\frac{j\alpha_0}{2k\pi}\right)} .
$$
We formulate the following conjecture:
\begin{quote}
\textit{If the Reeb periodic orbit is non degenerated, then for every $j\in\N\setminus\{0\}$ there exists a sequence of closed geodesics covering $j$ times $\Gamma$ with lengths close to $T_{j,k}$ for $k$ large enough.}
\end{quote}

This conjecture is consistent with our computation made in the Heisenberg flat case in \cite[Section 3.1]{CHT-I} and with the case of the sphere $S^3$ computed in \cite{KV-15}.
We guess that our conjecture holds true at least if $I$ is globally defined, which is the case in the example described in \cite{CHW-15}. 


\appendix

\section{Appendix}
\subsection{Sign conventions in symplectic geometry}\label{app:symplectic}
Since there are several possible sign conventions in the Hamiltonian formalism, we fix them as follows.

Given a smooth finite-dimensional manifold $M$, the canonical symplectic form on the cotangent bundle $T^*M$ is $\omega =dq \wedge dp = -d\theta$ with $\theta =p\, dq$ in local symplectic coordinates $(q,p)$.

Given a (smooth) Hamiltonian function $h$, the associated Hamiltonian vector field $\vec h$ is defined by $\iota_{\vec{h}} \omega = \omega(\vec h,\cdot) = dh$. In local coordinates, we have $\vec h = (\partial_p h,-\partial_x h)$. 

The Poisson bracket of two Hamiltonian functions $f$ and $g$ is defined by $\{ f,g \} = \omega(\vec f,\vec g)=df. \vec g = -dg.\vec f$. In local coordinates, we have $\{f,g\} = \partial_q f \partial_p g - \partial_p f \partial_q g$. We have $\overrightarrow{\{f,g\}} = - [\vec f,\vec g]$.

Given a vector field $X$ on $M$, the Hamiltonian lift is the function defined by $h_X(q,p) = \langle p,X(q)\rangle$. Given two vector fields $X$ and $Y$ on $M$, we have $\{h_X,h_Y\} = -h_{[X,Y]}$.

\subsection{A useful lemma for homogeneous polynomials}\label{app:linearalgebra}
Given any integer $k$, we define $P_k$ as the set of $k$-homogeneous polynomials in two variables $(u,v)$.
We define
\begin{equation*}
\begin{split}
P_k^0 &= \left\{ Q\in P_k\ \mid\ \int_0^{2\pi} Q(\cos\theta,\sin\theta)\, d\theta = 0 \right\} , \\
P_k^{\textrm{inv}} &= \{ c (u^2+v^2)^{k/2} \ \mid \ c\in\R\}.
\end{split}
\end{equation*}
The set $P_k^0$ is the set of $k$-homogeneous polynomials having zero average along the circle $u^2+v^2=1$.

We endow $P_k$ with the scalar product $\langle Q_1,Q_2\rangle = \frac{1}{2\pi} \int_0^{2\pi} Q_1(\cos\theta,\sin\theta)Q_2(\cos\theta,\sin\theta)\, d\theta$. Note that $P_k^0\perp P_k^{\textrm{inv}}$.
Considering polar coordinates, defining the endomorphism $A$ of $P_k$ defined by $A = u\partial_v - v\partial_u$ (or equivalently, considering the operator $\partial_\theta$), we have the following immediate result.

\begin{lemma}
\begin{itemize}
\item If $k$ is odd then $P_k = P_k^0$ and $A$ is invertible.
\item If $k$ is even then $P_k = P_k^0 \oplus P_k^{\textrm{inv}}$, and we have $\mathrm{range}(A) = P_k^0$ and $\ker(A)=P_k^{\textrm{inv}}$.
\end{itemize}
\end{lemma}

\begin{remark}\label{rem:linearalgebra}
Setting $I(u,v)=u^2+v^2$, endowing $\R^2_{u,v}$ with the symplectic form $du\wedge dv$, it is useful to note that, given any $Q\in P_k$, we have
$$
\{ I, Q \} = u\partial_v Q - v\partial_u Q = AQ.
$$
In particular, when $k$ is odd, we have $\{ I, Q \}=0$ if and only if $Q=0$.
\end{remark}



\begin{thebibliography}{99}

\bibitem{Mazzucchelli}
A. Abbondandolo, L. Macarini, M. Mazzucchelli, G. P. Paternain,
\textit{Infinitely many periodic orbits of exact magnetic flows on surfaces for almost every subcritical energy level},
to appear in J. Eur. Math. Soc.

\bibitem{yCdV-17}
Y. Colin de Verdi\`ere,
\textit{Magnetic fields and sub-Riemannian geometry},
In preparation (2017).

\bibitem{CHT-I}
Y. Colin de Verdi\`ere, L. Hillairet,  E. Tr\'elat,
\textit{Spectral asymptotics for sub-Riemannian Laplacians I: Quantum ergodicity and quantum limits in the 3D contact case},
Duke Math.
J., {\bf 167(1)}:109--174 (2018).


\bibitem{CHW-15}
Y. Colin de Verdi\`ere, J. Hilgert, T. Weich,
\textit{Irreducible representations of $SL_2(\R)$ and the Peyresq's operators},
Work in progress.

\bibitem{IS-02}
I. Iantchenko, J. Sj\"ostrand,
\textit{Birkhoff normal forms for Fourier Integral Operators II},
American J. Math. {\bf 124 (4)} (2002), 817--850.

\bibitem{KV-15}
D.  Klapheck, M. VanValkenburgh, 
\textit{The length spectrum of the sub-Riemannian three-sphere},
Preprint ArXiv:1507.03041 (2015).

\bibitem{Le-39} D. Lewis,
\textit{Formal power series transformations}
Duke Math. J. {\bf 5} (1939), 794--805. 

\bibitem{Me-84}
R. B. Melrose,
\textit{The wave equation for a hypoelliptic operator with symplectic characteristics of codimension two},
J. Analyse Math. {\bf 44} (1984-1985), 134--182.

\bibitem{Mo-02}
R. Montgomery,
\textit{A tour of subriemannian geometries, their geodesics and applications},
Mathematical Surveys and Monographs {\bf 91}, American Mathematical Society, Providence, RI, 2002.

\bibitem{Nelson}
E. Nelson,
\textit{Topics in dynamics I: flows},
Mathematical Notes, Princeton University Press, Princeton, N.J.; University of Tokyo Press, Tokyo, 1969, iii+118 pp.

\bibitem{Ra-VN-13}
N. Raymond,  S. V{\~u} Ng{\d o}c,
\textit{Geometry and spectrum in 2D magnetic wells},
Ann. Inst. Fourier (Grenoble) {\bf 65} (2015), no. 1, 137--169. 

\bibitem{St-50}
S. Sternberg,
\textit{Infinite Lie groups and formal aspects of dynamical systems},
J. Math. Mech. {\bf 10} (1961), 451--474. 



\end{thebibliography}
\end{document}